\documentclass[11pt, reqno]{article}
\usepackage{amsmath,amssymb,amsfonts,amsthm}
\usepackage{color}

\usepackage[colorlinks=true,linkcolor=blue,citecolor=blue,urlcolor=blue,pdfborder={0 0 0}]{hyperref}
\usepackage{cleveref}

\usepackage{caption}
\usepackage{thmtools}

\usepackage{mathrsfs}

\crefname{theorem}{Theorem}{Theorems}
\crefname{thm}{Theorem}{Theorems}
\crefname{lemma}{Lemma}{Lemmas}
\crefname{lem}{Lemma}{Lemmas}
\crefname{remark}{Remark}{Remarks}
\crefname{prop}{Proposition}{Propositions}
\crefname{defn}{Definition}{Definitions}
\crefname{claim}{Claim}{Claims}
\crefname{corollary}{Corollary}{Corollaries}
\crefname{conjecture}{Conjecture}{Conjectures}
\crefname{question}{Question}{Questions}
\crefname{chapter}{Chapter}{Chapters}
\crefname{section}{Section}{Sections}
\crefname{figure}{Figure}{Figures}

\theoremstyle{plain}
\newtheorem{thm}{Theorem}[section]

\newtheorem{corollary}[thm]{Corollary}

\newtheorem{prop}[thm]{Proposition}

\theoremstyle{definition}

\theoremstyle{remark}
\newtheorem{remark}[thm]{Remark}

\numberwithin{equation}{section}

\renewcommand{\P}{\mathbb P}
\newcommand{\E}{\mathbb E}

\newcommand{\R}{\mathbb R}
\newcommand{\Z}{\mathbb Z}
\newcommand{\N}{\mathbb N}

\newcommand{\cF}{\mathcal F}
\newcommand{\cG}{\mathcal G}

\usepackage[margin=1.13in]{geometry}
\usepackage{extarrows}
\usepackage{titling}
\usepackage{commath}
\usepackage{mathtools}
\usepackage{titlesec}
\newcommand{\eps}{\varepsilon}
\usepackage{bbm}
\usepackage{setspace}
\usepackage{enumitem}
\usepackage{tikz}

\usepackage{cite}

\newcommand{\lrDini}[1]{\left(\frac{d}{d #1}\right)_{\hspace{-0.2em}+}\!}

\newcommand{\bP}{\mathbf{P}}
\newcommand{\bE}{\mathbf{E}}

\newcommand{\Cov}{\operatorname{Cov}}
\newcommand{\CoVr}{\operatorname{CoVr}}
\newcommand{\phiw}{\phi^\mathrm{w}}
\newcommand{\phif}{\phi^\mathrm{f}}
\newcommand{\phih}{\phi^\#}

\newcommand{\hash}{\#}

\newcommand{\myfrac}[3][0pt]{\genfrac{}{}{}{}{\raisebox{#1}{$#2$}}{\raisebox{-#1}{$#3$}}}

\pretitle{\begin{flushleft}\Large}
\posttitle{\par\end{flushleft}\vskip 0.5em}
\preauthor{\begin{flushleft}}
\postauthor{
\\
\vspace{0.5em}
\footnotesize{Statslab, DPMMS, University of Cambridge.\\ Email: 
\href{mailto:t.hutchcroft@maths.cam.ac.uk}{t.hutchcroft@maths.cam.ac.uk}}
\end{flushleft}}
\predate{\begin{flushleft}}
\postdate{\par\end{flushleft}}
\title{\bf New critical exponent inequalities for percolation and the random cluster model}

\renewenvironment{abstract}
 {\par\noindent\textbf{\abstractname.}\ \ignorespaces}
 {\par\medskip}

\titleformat{\section}
  {\normalfont\large \bf}{\thesection}{0.4em}{}

\author{{\bf Tom Hutchcroft}}
\begin{document}

\date{\small{\today}}

\maketitle

\setstretch{1.1}

\begin{abstract}
We apply a variation on the methods of Duminil-Copin, Raoufi, and Tassion [\emph{Ann.\ Math.\ 2018}] to establish a new differential inequality applying to 
both Bernoulli percolation 
and the Fortuin-Kasteleyn random cluster model. This differential inequality has a similar form to that derived for Bernoulli percolation by Menshikov [\emph{Soviet Math.\ Dokl.\ 1986}] but with the important difference that it describes 
the distribution of the 
\emph{volume} of a cluster rather than of its radius. We apply this differential inequality to prove the following: 
\begin{enumerate}
	\item The critical exponent inequalities $\gamma \leq \delta-1$ and $\Delta \leq \gamma +1$ hold for percolation and the random cluster model on any transitive graph. These inequalities are new even in the context of Bernoulli percolation on $\Z^d$, and are saturated in mean-field for Bernoulli percolation and for the random cluster model with $q \in [1,2)$. 
	\item  The volume of a cluster has an exponential tail in the entire subcritical phase of the random cluster model on any transitive graph. This proof also applies to infinite-range models, where the result is new even in the Euclidean setting.
\end{enumerate}
\end{abstract}



\section{Introduction}
\label{sec:intro}
\emph{Differential inequalities} play a central role in the rigorous study of percolation and other random media. Indeed, one of the most important theorems in the theory of Bernoulli percolation is that the phase transition is \emph{sharp}, meaning (in one precise formulation) that the radius of the cluster of the origin has an exponential tail throughout the entire subcritical phase.  This theorem was first proven in independent works of Menshikov \cite{MR852458} and Aizenman and Barsky \cite{aizenman1987sharpness}. While these two proofs were rather different, they both relied crucially on differential inequalities: In Menshikov's case this differential inequality was
\begin{equation}
\label{eq:Menshikov}
\frac{d}{dp}\log \bP_p(R \geq n) \geq  \frac{1}{p} \left[\frac{n}{\sum_{m=0}^n \bP_p(R \geq m)} - 1 \right]  \qquad \text{ for each } n \geq 1
\end{equation}
where $R$ denotes the radius of the cluster of the origin, while for Aizenman and Barsky the relevant differential inequalities were 
\begin{equation}
\label{eq:AizenmanBarsky}
M \leq h \frac{\partial M}{\partial h} + M^2 + p M \frac{\partial M}{\partial p}  \qquad \text{ and } \qquad \frac{\partial M}{\partial p}  \leq d M \frac{\partial M}{\partial h}, 
\end{equation}
where we write $|K|$ for the volume of the cluster of the origin, write $M=M_{p,h}$ for the \textbf{magnetization} $M=\bE_p[1-e^{-h|K|}]$, and write $d$ for the degree of the graph. 
An alternative, simpler proof of sharpness for percolation, which also relies on differential inequalities, was subsequently found by Duminil-Copin and Tassion \cite{duminil2015new}. 
Aside  from their use to establish  sharpness, the differential inequalities \eqref{eq:Menshikov} and \eqref{eq:AizenmanBarsky} also yield further  quantitative information about percolation at and near criticality. In particular, both inequalities can be used to derive bounds on \emph{critical exponents} associated to percolation; this is discussed further in \cref{subsec:percolation_exponents} and reviewed in detail in \cite{grimmett2010percolation}. Similar methods have also yielded similar results for the Ising model \cite{aizenman1987phase,duminil2015new}.



Aside from percolation and the Ising model, the class of models that were rigorously proven to undergo sharp phase transitions was, until recently, very limited. In particular, the derivations of both \eqref{eq:Menshikov} and  \eqref{eq:AizenmanBarsky}  rely heavily on the van den Berg-Kesten (BK) inequality \cite{MR799280}, and  are therefore rather specific to Bernoulli percolation. This situation has now improved drastically following the breakthrough work of Duminil-Copin, Tassion, and Raoufi \cite{duminil2017sharp}, who showed that the theory  of \emph{randomized algorithms} can often be used
 to prove sharpness of the phase transition in models satisfying the FKG lattice condition. They first applied this new methodology to prove that a differential inequality essentially equivalent to that of Menshikov \eqref{eq:Menshikov} holds for the Fortuin-Kasteleyn random-cluster model (with $q\geq 1$), from which they deduced sharpness of the phase transition for this model and the ferromagnetic Potts model. Variations on their methods have subsequently been used to prove sharpness results for several other models, including Voronoi percolation \cite{duminil2017exponential}, Poisson-Boolean percolation \cite{duminil2018subcritical}, the Widom-Rowlinson model \cite{dereudre2018sharp}, level sets of smooth planar Gaussian fields \cite{muirhead2018sharp}, and the contact process \cite{beekenkamp2018sharpness}.


The main new technical tool introduced by \cite{duminil2017sharp} was a generalization of the \emph{OSSS inequality} from product measures to \emph{monotonic measures}. This inequality, introduced by O'Donnel, Saks, Schramm, and Servedio \cite{o2005every}, can be used to derive differential inequalities for percolation in the following way: 
Let $A$ be an event depending on at most finitely many edges, and suppose that we have an algorithm for computing whether or not $A$ occurs. This algorithm  decides sequentially which edges to reveal the status of, with decisions depending on what it has previously seen and possibly also some external randomness, stopping when it has determined whether or not $A$ has occurred. 
For each edge $e$, let $\delta_e$ be the \textbf{revealment} of $e$, defined to be the probability that the status of the edge $e$ is ever queried by the algorithm. Then the OSSS inequality implies that
\begin{equation}
\label{eq:OSSSintro1}
\frac{d}{dp} \log \bP_p(A) \geq \frac{1-\bP_p(A)}{p(1-p) \max_{e\in E} \delta_e}.
\end{equation}
In particular, if $\bP_p(A)$ is not too large and there exists a randomized algorithm determining whether or not $A$ holds with low maximum revealment, then the logarithmic derivative of $\bP_p(A)$ is large. This yields an extremely flexible methodology for deriving differential inequalities for percolation. Even greater flexibility is provided by the \emph{two-function} version of the OSSS inequality, which implies in particular that if $A$ and $B$ are increasing events and we have some randomized algorithm that determines whether or not $B$ occurs, then
\begin{equation}
\label{eq:OSSSintro2}
\frac{d}{dp} \log \bP_p(A) \geq \frac{\bP_p(B \mid A)-\bP_p(B)}{p(1-p) \max_{e\in E} \delta_e}.
\end{equation}



\medskip

\textbf{The new differential inequality.}
In this article, we apply the OSSS inequality to establish a new differential inequality for percolation and the random cluster model. Once we establish this inequality, we use it to prove several other new results for these models which are detailed in the following subsections. Our new  inequality is similar to Menshikov's inequality \eqref{eq:Menshikov} but describes the distribution of the \emph{volume} of a cluster rather than of its radius. In the case of percolation on a transitive graph, we obtain in particular that
\begin{equation}
\label{eq:percdiffineq}
 \frac{d}{dp} \log \bP_p\left(|K| \geq n \right)
\geq \frac{1}{2p(1-p)} \left[ \frac{(1- e^{-\lambda})n}{\lambda \sum_{m=1}^{\lceil n/\lambda \rceil }\bP_p(|K|\geq m)}-1\right]
\end{equation}
for each $n\geq 0$, $\lambda > 0$, and $0\leq p < 1$, where $K$ is the cluster of some vertex $v$ and $|K|$ is the number of vertices it contains.  We will typically apply this inequality with $\lambda =1$, but the freedom to change $\lambda$ is sometimes useful for optimizing constants.

\medskip

We derive \eqref{eq:percdiffineq} by introducing a \textbf{ghost field}  as in \cite{aizenman1987sharpness}, i.e., an independent Bernoulli process $\cG$ on the vertices of $G$ such that $\cG(v)=1$ with probability $1-e^{-\lambda/n}$ for each vertex $v$ of $G$. We call vertices with $\cG(v)=1$ \textbf{green}. We then apply the two-function OSSS inequality where $A$ is the event that $|K|\geq n$, $B$ is the event that $K$ includes a green vertex, and our algorithm simply examines the ghost field at every site and then explores the cluster of each green vertex it discovers. 

\medskip

In the remainder of the introduction we describe consequences of the differential inequality \eqref{eq:percdiffineq} and of its generalization to the random cluster model.

\subsection{Critical exponent inequalities for percolation}
\label{subsec:percolation_exponents}

In this section we discuss the applications of our differential inequality \eqref{eq:percdiffineq} to rigorously establish inequalities between critical exponents in percolation. 
We first recall the definition of Bernoulli bond percolation, referring the reader to e.g.\ \cite{grimmett2010percolation} for further background. Let $G=(V,E)$ be a connected, locally finite, transitive graph, such as the hypercubic lattice $\Z^d$. Here, \textbf{locally finite} means that every vertex has finite degree, and \textbf{transitive} means that for any two vertices $x$ and $y$ of $G$, there is an automorphism of $G$ mapping $x$ to $y$. In \textbf{Bernoulli bond percolation}, each edge of $G$ is either deleted (closed) or retained (open) independently at random with retention probability $p\in [0,1]$ to obtain a random subgraph $\omega_p$ of $G$. The connected components of $\omega_p$ are referred to as \textbf{clusters}. We write $\bP_p$  and $\bE_p$ for probabilities and expectations taken with respect to the law of $\omega_p$. 

It is expected that the behaviour of various quantities describing percolation at and near the critical parameter 
\[
p_c = \inf\bigl\{ p \in [0,1] : \omega_p \text{ has an infinite cluster a.s.}\bigr\}
\] are described by \emph{critical exponents}. For example, it is predicted that for each $d\geq 2$ there exist exponents $\beta,\gamma,\delta,$ and $\Delta$ such that percolation on $\Z^d$ satisfies
\begin{align*}
\bP_p(|K|=\infty) &\approx (p-p_c)^\beta & \text{ as } p &\downarrow p_c\\
\bE_p\left[|K|\right] &\approx (p_c-p)^{-\gamma} & \text{ as } p &\uparrow p_c\\
\bP_{p_c}\left(|K| \geq n\right) &\approx n^{-1/\delta} & \text{ as } n &\uparrow \infty\\
\bE_p\left[|K|^k\right] &\approx (p_c-p)^{-(k-1)\Delta+\gamma} & \text{ as } p &\uparrow p_c,
\end{align*}
where $K$ is the cluster of the origin and $\approx$ means that the ratio of the logarithms of the two sides tends to $1$ in the appropriate limit.
 Proving the existence of and computing these critical exponents is considered to be a central problem in mathematical physics. While important progress has been made in two dimensions \cite{MR879034,smirnov2001critical,smirnov2001critical2,lawler2002one}, in high dimensions \cite{MR1043524,MR762034,MR1127713,MR923855,fitzner2015nearest}, and in various classes of infinite-dimensional graphs \cite{MR1833805,MR1888869,1804.10191,Hutchcroftnonunimodularperc}, the entire picture remains completely open in dimensions $3 \leq d \leq 6$.

A further central prediction of the non-rigorous theory is that, if they exist, these exponents should satisfy  the \textbf{scaling relations}
\begin{equation}
\label{eq:scalingrelations}
\gamma = \beta(\delta-1) \qquad \text{ and } \qquad \beta \delta = \Delta
\end{equation}
in every dimension. (There are also two further scaling relations involving the exponents $\alpha$, $\nu$, and $\eta$, which we have not introduced.) See \cite{grimmett2010percolation} for a heuristic derivation of these exponents for mathematicians and e.g.\ \cite{cardy1996scaling} for more physical derivations.
The heuristic derivations of \eqref{eq:scalingrelations} do not rely on any special features of percolation, and the scaling relations \eqref{eq:scalingrelations} are expected to hold for any natural model of random media undergoing a continuous phase transition.

A rigorous proof of \eqref{eq:scalingrelations} remains elusive. Special cases in which progress has been made include the two-dimensional case, where the scaling relations \eqref{eq:scalingrelations} were proven by Kesten \cite{MR879034},  
and the high-dimensional case,  where it has been proven rigorously \cite{MR762034,MR1043524,MR923855,fitzner2015nearest} that the exponents take their  \emph{mean-field} values of $\beta=1,\gamma=1, \delta=2$, and $\Delta=2$, from which it follows that \eqref{eq:scalingrelations} holds. (See e.g.\ \cite{fitzner2015nearest,MR2239599} for a detailed overview of what is known in high-dimensional percolation.) See also \cite{MR3940769} for related results on two-dimensional Voronoi percolation. 
Aside from this, progress on the rigorous understanding of \eqref{eq:scalingrelations} has been limited to proving \emph{inequalities} between critical exponents. In particular, it is known that
\begin{equation}
\label{eq:oldinequalities}
1 \leq \beta (\delta -1), \quad \beta \delta \geq 2, \quad  \frac{\gamma\delta }{\delta -1} \geq 2, \quad  \frac{\gamma  \delta }{\delta -1} \leq \Delta, \quad \text{ and } \quad 2 \gamma \geq \Delta
\end{equation}
whenever these exponents are well-defined: The first of these inequalities is due to Aizenman and Barsky \cite{aizenman1987sharpness}, the second, third, and fourth are due to Newman \cite{MR869320,MR912497,MR894551}, and the fifth is due to Aizenman and Newman \cite{MR762034}. All of these inequalities are saturated when the exponents take their mean-field values, and the fourth is expected to be an equality in every dimension. 
These inequalities are complemented by the \emph{mean-field bounds}
\begin{equation}
\label{eq:meanfieldbounds}
\beta \leq 1, \qquad \gamma \geq 1, \qquad \delta \geq 2, \qquad \text{and} \qquad  \Delta \geq 2
\end{equation}
which were first proven  to hold by Chayes and Chayes \cite{MR894542}, Aizenman and Newman \cite{MR762034}, Aizenman and Barsky \cite{aizenman1987sharpness}, and Durrett and Nguyen \cite{durrett1985thermodynamic} respectively. 
 See \cite[Chapters 9 and 10]{grimmett2010percolation} for further details, and  \cite{MR852458,MR912497,MR869320,duminil2015new,duminil2017sharp} for alternative proofs of some of these inequalities.

Our first application of the differential inequality \eqref{eq:percdiffineq} is to rigorously prove two new critical exponent inequalities, namely that
\begin{equation}
\label{eq:newinequalities}
\gamma \leq \delta -1 \qquad \text{ and } \qquad \Delta \leq \gamma + 1.
\end{equation}
Note that the inequalities of \eqref{eq:newinequalities} are consistent with the conjectural scaling relations \eqref{eq:scalingrelations} due to the mean-field bound $\beta \leq 1$, and are saturated when the relevant exponents take their mean-field values. The first of these inequalities is particularly interesting as it points in a different direction to the previously known inequalities given in \eqref{eq:oldinequalities}. 

We will deduce \eqref{eq:newinequalities} as a corollary of the following two theorems, which are derived from \eqref{eq:percdiffineq} and which give more precise quantitative versions of these critical exponent inequalities. 
The first of these theorems relates the distribution of the volume of a critical cluster to the distribution of the volume of a subcritical cluster. It implies the critical exponent inequalities $\gamma \leq \delta -1$ and $\Delta \leq \delta$. Recall that we write $K$ for the cluster of some arbitrarily chosen vertex.
 
\begin{thm}
\label{thm:gammadelta_percolation} Let $G$ be an infinite, connected, locally finite transitive graph, and suppose that there exist constants $C>0$ and $\delta >1$ such that
\[
\bP_{p_c}\left(|K| \geq n\right) \leq C n^{-1/\delta}
\]
for every $n\geq 1$. Then the following hold:
\begin{enumerate}
	\item
There exist positive constants $c$ and $C'$ such that
\[
\bP_p\left(|K| \geq n\right) \leq  C'n^{-1/\delta} \exp\left[ -c(p_c-p)^{\delta}n \right] 
\]
for every $0 \leq p < p_c$ and $n\geq 1$.
	\item There exists a constant $C''$ such that
	\[
\bE_p\left[|K|^k\right] \leq k!\left[\frac{C''}{p_c-p}\right]^{(\delta-1) + (k-1)\delta } 
\]
for every  and $0\leq p < p_c$ and $k\geq 1$.
\end{enumerate}
\end{thm}



The next theorem bounds the growth of the $k$th moment of the cluster volume as $p \uparrow p_c$ in terms of the growth of the first moment as $p\uparrow p_c$. It implies the critical exponent inequality $\Delta \leq \gamma + 1$.

\begin{thm}
\label{thm:gammagap_percolation}
Let $G$ be an infinite, connected, locally finite transitive graph, and suppose that there exist constants $C>0$ and $\gamma \geq 1$ such that
\[
\bP_p\left[|K|\right] \leq C (p_c-p)^{-\gamma}
\]
for every $n\geq 1$ and $0\leq p <p_c$. Then there exists a constant $C'$ such that
\[
\bE_p\left[|K|^k\right] \leq k! \left[\frac{C'}{p_c-p}\right]^{\gamma+(k-1)(\gamma+1)} 
\]
for every $0 \leq p <  p_c$ and $k\geq 1$.
\end{thm}

In light of the results of \cite{hutchcroft2018locality}, \cref{thm:gammadelta_percolation} also has the following consequence for  percolation on unimodular transitive  graphs of exponential growth. Here, the \textbf{growth} of a transitive graph $G$ is defined to be $\operatorname{gr}(G) = \lim_{n\to\infty}|B(v,n)|^{1/n}$ where $v$ is a vertex of $G$ and $|B(v,n)|$ is the ball of radius $n$ around $v$. See \cite{hutchcroft2018locality,1804.10191} for more on what is known concerning percolation on such transitive graphs.

\begin{corollary}
\label{corollary:polybound}
For every $g>1$ and $M<\infty$ there exist constants $C=C(g,M)$ and $A=A(g,M)$ such that for every unimodular transitive  graph $G$ with degree at most $M$ and $\operatorname{gr}(G)\geq g$, the bound
\begin{equation}
\label{eq:polybound}
\bE_{p} \left[|K|^k\right] \leq C (p_c-p)^{-A k}
\end{equation}
holds for every $0 \leq p < p_c$ and $k\geq 1$.
\end{corollary}

\subsection{The random cluster model}
\label{subsec:introFK}

In this section we discuss generalizations to and applications of \eqref{eq:percdiffineq} to the random cluster model (a.k.a.\ FK-percolation). Since its introduction by Fortuin and Kesteleyn \cite{MR0359655}, the random cluster model has become recognized as the archetypal example of a dependent percolation model, and is closely connected to the Ising and Potts models. We refer the reader to \cite{GrimFKbook} for further background on the model.  We expect that the results in this section will also generalize to other models for which sharpness has been proven via the methods of \cite{duminil2017sharp}.

We begin by defining the random cluster model, which we do at the natural generality of \emph{weighted graphs}. We will take a slightly unconventional approach to allow for a unified treatment of short- and long-range models.
In this paper, a \textbf{weighted graph} $G=(G,J)$ is defined to be a countable graph $G=(V,E)$ together with an assignment of positive \textbf{coupling constants} $\{J_e : e \in E\}$ such that for each vertex of $G$, the sum of the coupling constants $J_e$ over all $e$ adjacent to $v$ is finite. A graph automorphism of $G$ is a weighted graph automorphism of $(G,J)$ if it preserves the coupling constants, and a weighted graph is said to be \textbf{transitive} if for every $x,y \in V$ there is an automorphism sending $x$ to $y$. 
Note that our weighted graphs are \emph{not} required to be locally finite. 

Let $(G=(V,E),J)$ be a weighted graph with $V$ finite, so that $\sum_{e\in E} J_e <\infty$. (Since we did not assume that $G$ is simple, it is possible for the edge set to be infinite.) For each $q >0$ and $\beta \geq 0$, we let the \textbf{random cluster measure} $\phi_{G,\beta,q}$ be the purely atomic probability measure on $\{0,1\}^E$ defined by
\[
\phi_{G,\beta,q}(\{\omega\}) = \frac{1}{Z_{G,\beta,q}} q^{\hash\mathrm{clusters}(\omega)} \prod_{e\in E} (e^{\beta J_e}-1)^{\omega(e)},
\]
where $Z_{G,\beta,q}$ is a normalizing constant. In particular, $\phi_{G,\beta,q}$ is supported on configurations containing at most finitely many edges. It is easily verified that this measure is well-defined under the above hypotheses, that is, that $Z_{G,\beta,q}<\infty$.
If $q=1$ and $J_e \equiv 1$, the measure $\phi_{G,\beta,q}$ is simply the law of Bernoulli bond percolation with retention probability $\beta = -\log (1-p)$. Similarly, if $q=1$ and the coupling constants are non-constant then the measure $\phi_{G,\beta,q}$ is the law of \emph{inhomogeneous Bernoulli bond percolation}.

Now suppose that $G$ is an infinite weighted graph. For each $q\geq 1$, we define the \textbf{free} and \textbf{wired} random cluster measures $\phif_{\beta,q}$ and $\phiw_{\beta,q}$ on $G$ by taking limits along finite subgraphs of $G$ with either free or wired boundary conditions. Let $(V_n)_{n\geq 1}$ be an increasing sequence of finite subsets of $V$ with $\bigcup_{n\geq 1} V_n = V$. For each $n\geq 1$, we define $G_n$ to be the subgraph of $G$ induced by $V_n$ and let $G_n^*$ be the graph obtained by identifying all vertices in $V \setminus V_n$ and deleting all self-loops that are created. Both $G_n$ and $G_n^*$ inherit the coupling constants of $G$ in the natural way. It is shown in \cite[Chapter 4]{GrimFKbook} that if $q\geq 1$ and $\beta \geq 0$ then the weak limits
\begin{align*}
\phif_{G,\beta,q}:= \mathop{\operatorname{w-lim}}_{n\to\infty} \phi_{G_n,\beta,q} \qquad \text{ and } \qquad \phiw_{G,\beta,q}:= \mathop{\operatorname{w-lim}}_{n\to\infty}\phi_{G^*_n,\beta,q} 
\end{align*}
are well-defined and do not depend on the choice of exhaustion $(V_n)_{n\geq 1}$ for every $q\geq 1$ and $n\geq 1$. (It is not known whether these infinite volume limits are well-defined when $q<1$.) From now on we will drop the $G$ from our notation and write simply $\phif_{\beta,q}$ and $\phiw_{\beta,q}$. Note that $\phiw_{\beta,q}$ stochastically dominates $\phif_{\beta,q}$ for each fixed $\beta\geq 0$ and $q\geq 1$, and that for each $q\geq 1$, $\hash \in \{\mathrm{w},\mathrm{f}\}$, and $0\leq \beta_1 \leq \beta_2$, the measure $\phih_{\beta_2,q}$ stochastically dominates $\phih_{\beta_1,q}$.

The generalization of the differential inequality \eqref{eq:percdiffineq} to the random cluster model may be stated as follows. Here  $\lrDini{\beta}$ denotes the \emph{lower-right Dini derivative}, which we introduce properly in \cref{subsec:derivatives}.
\begin{prop}
\label{cor:FKdiffineq}
Let $(G,J)$ be an infinite transitive weighted graph, and let $q\geq 1$ and $\# \in \{\mathrm{f},\mathrm{w}\}$. Then 
\begin{equation}
\label{eq:FKdiffineq}
\max_{e\in E} \left[\frac{e^{\beta J_e}-1}{J_e}\right]
   \lrDini{\beta} \log \phih_{\beta,q}\left(|K|\geq n\right) 
\geq \frac{1}{2} \left[ \frac{(1- e^{-\lambda})n}{\lambda \sum_{m=1}^{\lceil n/\lambda \rceil}\phih_{\beta,q}(|K|\geq m)}-1\right]
\end{equation}
for every $\beta \geq 0$, $\lambda >0$, and $n\geq 1$.
\end{prop}

Our main application of \cref{cor:FKdiffineq} is to establish the following sharpness result for the random cluster model.
For each $\hash\in \{\mathrm{f},\mathrm{w}\}$ the \textbf{critical parameter} $\beta_c^\hash$ is defined to be
\[
\beta_c^\hash = \beta_c^\hash(G,q) = \inf\left\{\beta \geq 0 : \phih_{G,\beta,q}(|K_v|=\infty)>0 \text{ for some $v\in V$} \right\}.
\]
We always have that $\beta_c^\mathrm{w} \leq \beta_c^\mathrm{f}$ by stochastic domination. It is known that $\beta_c^\mathrm{w}=\beta_c^\mathrm{f}$ for the random cluster model on $\Z^d$ and other transitive amenable graphs, while it is believed that strict inequality should hold for $q>2$ in the nonamenable case \cite[Chapter 10]{GrimFKbook}. It was shown in \cite{duminil2017sharp} that the following holds for every connected, locally finite, transitive graph, every $q\geq 1$, and every $\hash \in\{\mathrm{f},\mathrm{w}\}$: 
\begin{enumerate}
\item If $\beta < \beta_c^\hash$ then there exist positive constants $C_\beta,c_\beta$ such that
\[
\phih_{\beta,q}(R \geq n) \leq C_\beta e^{-c_\beta n}
\]
for every $n\geq 1$, where $R$ is the radius of the cluster of some fixed vertex $v$  
as measured by the graph metric on $G$.
\item There exists a constant $c$ such that 
\[
\phih_{\beta,q}(|K| =\infty ) \geq c(\beta-\beta_c^\hash)
\]
for every $\beta> \beta_c^\hash$ with $\beta-\beta_c^\hash$ sufficiently small. 
\end{enumerate}
The following theorem improves this result by establishing an exponential tail for the \emph{volume} rather than the radius and also by applying to long-range models, which were not treated by \cite{duminil2017sharp}. (Note that in the case of finite-range models on $\Z^d$, the results of \cite{duminil2017sharp} were known to imply an exponential tail on the volume by earlier conditional results \cite[Section 5.6]{GrimFKbook}.)

\begin{thm}
\label{thm:sharpness}
 Let $(G,J)$ be an infinite transitive weighted graph. Let $q\geq 1$ and $\# \in \{\mathrm{f},\mathrm{w}\}$. Then the following hold.
 \begin{enumerate}
 \item For every $0\leq \beta < \beta_c^\#$ there exist positive constants $C_\beta$ and $c_\beta$ such that
\[
\phih_{\beta,q}(|K|\geq n) \leq C_\beta e^{-c_\beta n}
\]
for every $n\geq 1$.
\item
The inequality
\[
\phih_{\beta,q}(|K|=\infty) \geq \frac{\beta-\beta_c^\hash}{2\max_{e\in E} \left[\frac{e^{\beta J_e}-1}{J_e}\right]+\beta-\beta_c^\hash}
\]
holds for every $\beta>\beta_c^\hash$.
\end{enumerate}
\end{thm}

An immediate corollary of \cref{thm:sharpness} is that $\phih_{\beta,q}[|K|]<\infty$ for every $\beta<\beta_c^\hash$ under the same hypotheses, which did not follow from the results of \cite{duminil2017sharp} in the case that the $G$ has exponential volume growth. This allows us to apply the method of \cite{Hutchcroft2016944} and the fact that $\phif_{\beta,q}$ is weakly left-continuous in $\beta$  for each $q\geq 1$ \cite[Proposition 4.28c]{GrimFKbook} to deduce the following corollary for the random cluster model on transitive graphs of exponential growth. This adaptation has already been carried out in the case $q=2$ (the FK-Ising model) by Raoufi \cite{1606.03763}.

\begin{corollary}
\label{cor:Fekete}
Let $G$ be a connected, locally finite, transitive graph of exponential growth and let $q\geq 1$. Then $\phif_{\beta_c^\mathrm{f},q}(|K|=\infty)=0$.
\end{corollary}


Finally, we generalize of \cref{thm:gammadelta_percolation,thm:gammagap_percolation} to the random cluster model. 
\begin{thm}
\label{thm:gammadelta_FK} Let $(G,J)$ be an infinite transitive weighted graph. Let $\beta_0>0$, $q\geq 1$, and $\# \in \{\mathrm{f},\mathrm{w}\}$, and suppose that there exist constants $C>0$ and $\delta >1$ such that
\[
\phih_{\beta_0,q}\left(|K| \geq n\right) \leq C n^{-1/\delta}
\]
for every $n\geq 1$. Then the following hold:
\begin{enumerate}
	\item
There exist positive constants $c_1$ and $C_1$ such that
\[
\phih_{\beta,q}\left(|K| \geq n\right) \leq  C_1 n^{-1/\delta} \exp\left[ -c_1 (\beta_0-\beta)^\delta n \right] 
\]
for every $n\geq 1$ and $0 \leq \beta < \beta_0$.
	\item There exists a positive constant $c_2$ such that
	\[
\phih_{\beta,q}\left[|K|^k\right] \leq k! \left[c_2(\beta_0-\beta)\right]^{-\delta k +1} 
\]
for every $k\geq 1$ and $0\leq \beta < \beta_0$.
\end{enumerate}
\end{thm}

\begin{thm}
\label{thm:gammagap_FK}
Let $(G,J)$ be an infinite transitive weighted graph. Let $\beta_0>0$, $q\geq 1$, and $\# \in \{\mathrm{f},\mathrm{w}\}$, and suppose that there exist constants $C>0$ and $\delta >1$ such that
\[
\phih_{\beta,q}\left[|K|\right] \leq C (\beta_0-\beta)^{-\gamma}
\]
for every $n\geq 1$. Then there exists a positive constant $c$ such that
\[
\phih_{\beta,q}\left[|K|^k\right] \leq k! \left[c(\beta_0-\beta)\right]^{-(k-1)(\gamma+1)- \gamma} 
\]
for every $n,k\geq 1$ and $0 \leq \beta <\beta_0$.
\end{thm}

It follows from these theorems that the critical exponent inequalities $\gamma \leq \delta -1$ and $\Delta \leq \gamma + 1$ hold for the random cluster model whenever these exponents are  well-defined and $q\geq 1$. 

We remark that the previous literature on critical exponents for the random cluster model with $q\notin \{1,2\}$ seems rather limited, although the sharpness results of \cite{duminil2017sharp} imply the mean-field bound $\beta \leq 1$.  It is also known that the exponent inequalities we derive here are sharp in mean-field for $q\in [1,2)$, where the exponents are the same as for percolation \cite{bollobas1996random}. 
Note that it is expected that when $q$ is large the random cluster model undergoes a \emph{discontinuous} (first-order)  phase transition, see \cite{bollobas1996random,laanait1991interfaces,duminil2016discontinuity,ray2019short} and references therein. 

\section{Background}

\subsection{Monotonic measures}

Let $A$ be a countable set. A probability measure $\mu$ on $\{0,1\}^A$ is said to be \textbf{positively associated} if
\[
\mu(f(\omega)g(\omega)) \geq \mu(f(\omega))\mu(g(\omega))
\]
for every pair of increasing functions $f,g: \{0,1\}^A \to \R$, and is said to be
\textbf{monotonic} if 
\[
\mu\left(\omega(e)=1 \mid \omega|_F = \xi\right) \geq \mu\left(\omega(e)=1 \mid \omega|_F=\zeta\right)
\]
whenever $F \subset A$, $e\in A$, and $\xi,\zeta \in \{0,1\}^F$ are such that $\xi \geq \zeta$. It follows immediately from this definition that if $\mu$ is a monotonic measure on $\{0,1\}^A$ and $\nu$ is a monotonic measure on $\{0,1\}^B$, then the product measure $\mu \otimes \nu$ is monotonic on $\{0,1\}^{A \amalg B}$. 

Monotonic measures are positively associated, but positively associated measures need not be monotonic, see \cite[Chapter 2]{GrimFKbook}. Indeed, it is proven in \cite[Theorem 2.24]{GrimFKbook} that if $A$ is finite and $\mu$ gives positive mass to every element of $\{0,1\}^A$ then it is monotonic if and only if it satisfies the \textbf{FKG lattice condition}, which states that
\[
\mu(\omega_1 \vee \omega_2)\mu(\omega_1 \wedge \omega_2) \geq \mu(\omega_1)\mu(\omega_2)
\]
for every $\omega_1,\omega_2 \in \{0,1\}^A$. In particular, it follows readily from this that the random cluster measures with $q\geq 1$ on  any (finite or countably infinite) weighted graph $(G,J)$ are monotonic. 

\subsection{Derivative formulae}
\label{subsec:derivatives}

Let $G=(G,J)$ be a \emph{finite} weighted graph. Then for every function $F:\{0,1\}^E \to \R$, we have the derivative formula  \cite[Theorem 3.12]{GrimFKbook}
\begin{equation}
\label{eq:DerivativeFormula}
\frac{d}{d\beta} \phi_{\beta,q}\left[F(\omega)\right] = \sum_{e\in E} \frac{J_e}{e^{\beta J_e}-1} \Cov_{\phi_{\beta,q}}\left[F(\omega),\omega(e)\right],
\end{equation}
where we write $\Cov_\mu[X,Y]=\mu(XY)-\mu(X)\mu(Y)$ for the covariance of two random variables $X$ and $Y$ under the measure $\mu$.

To discuss the generalization of this derivative formula to the infinite volume case,  we must first introduce \emph{Dini derivatives}, referring the reader to \cite{MR1390758} for further background. The \textbf{lower-right Dini derivative} of a function $f:[a,b] \to \R$ is defined to be
\[
\lrDini{x}f(x) =  \liminf_{\eps \downarrow  0} \frac{f(x+\eps)-f(x)}{\eps}
\]
for each $x \in [a,b)$.  
Note that if $f:[a,b]\to \R$ is increasing then we have that
\[
f(b) - f(a) \geq \int_{a}^b \lrDini{x}f(x)  \dif x,
\]
so that we may use differential inequalities involving Dini derivatives in essentially the same way that we use standard differential inequalities. (It is a theorem of Banach \cite[Theorem 3.6.5]{MR1390758} that measurable functions have measurable Dini derivatives, so that the above integral is well-defined.) We also have the validity of the usual logarithmic derivative formula
\[
\lrDini{x} \log f(x) = \frac{1}{f(x)}\lrDini{x}f(x).
\]

The following proposition yields a version of \eqref{eq:DerivativeFormula} valid in the infinite-volume setting.
\begin{prop}
\label{prop:DiniDerivativeFormula}
Let $G=(G,J)$ be a weighted graph and let $F:\{0,1\}^E \to \R$ be an increasing function, and let $\hash\in \{\mathrm{f},\mathrm{w}\}$. Then
\begin{equation}
\label{eq:DiniDerivativeFormula}
\lrDini{\beta} \phih_{\beta,q}\left[F(\omega)\right] \geq \sum_{e\in E} \frac{J_e}{e^{\beta J_e}-1} \Cov_{\phih_{\beta,q}}\left[F(\omega),\omega(e)\right]
\end{equation}
for every $\beta \geq 0$.
\end{prop}


\begin{proof}
We prove the claim in the case $\hash=\mathrm{f}$,  the case $\hash=\mathrm{w}$  being similar.
Fix $\beta_0\geq 0$ and let $A$ be a finite set of edges. Let $(V_n)_{n\geq 1}$ be an  exhaustion of $V$, let $G_n$ be the subgraph of $G$ induced by $V_n$  and let $E_n$ be the edge set of $G_n$.  For each  $n\geq 1$, $\alpha,\beta \geq 0$ and $q\geq 1$ we define $\phi_{G_n,\beta,\beta_0,q,A}(\{\omega\})$  by
\[
\phi_{G_n,\beta,\beta_0,q,A}(\{\omega\}) =  \frac{1}{Z} q^{\hash \text{clusters} (\omega)}\prod_{e\in A \cap E_n} (e^{\beta J_e}-1)^{\omega(e)}\prod_{e\in E_n \setminus A} (e^{\beta_0 J_e}-1)^{\omega(e)}
\]
for an appropriate normalizing constant $Z=Z(n,\beta,q,A)$. The usual proof of the existence of the infinite volume random cluster measures yields that the measures $\phi_{G_n,\beta,\beta_0,q,A}$ converge weakly to a limiting measure $\phif_{\beta,\beta_0,q,A}$ as $n\to \infty$. Using the assumption that $A$ is finite, it is straightforward to adapt the usual proof of \eqref{eq:DerivativeFormula} to show that $\phif_{\beta,\beta_0,q,A}\left[F(\omega)\right]$ is differentiable and that
\[
\frac{d}{d\beta} \phif_{\beta,\beta_0,q,A}\left[F(\omega)\right] = \sum_{e\in A} \frac{J_e}{e^{\beta J_e}-1} \Cov_{\phif_{\beta,\beta_0,q,A}}\left[F(\omega),\omega(e)\right]
\]
for every function $F(\omega)\to \R$ with $\phif_{\beta_0,q}\left[F(\omega)\right]<\infty$. On the other hand, the FKG property implies that $\phif_{\beta,q}$ stochastically dominates $\phif_{\beta,\beta_0,q,A}$ for every $\beta \geq \beta_0$, and we deduce that if $F$ is increasing then
\begin{align*}
\liminf_{\beta \downarrow \beta_0} \frac{\phif_{\beta,q}\left[F(\omega)\right] - \phif_{\beta_0,q}\left[F(\omega)\right]}{\beta-\beta_0} & \geq \liminf_{\beta \downarrow \beta_0} \sup_A \frac{1}{\beta-\beta_0}\left[\phif_{\beta,\beta_0,q,A}\left[F(\omega)\right] - \phif_{\beta_0,q,A} \left[F(\omega)\right] \right]\\&\geq \sup_A \liminf_{\beta \downarrow \beta_0}  \frac{1}{\beta-\beta_0}\left[\phif_{\beta,\beta_0,q,A}\left[F(\omega)\right] - \phif_{\beta_0,q,A} \left[F(\omega)\right] \right]\\ &= \sup_A \sum_{e\in A} \frac{J_e}{e^{\beta J_e}-1} \Cov_{\phif_{\beta_0,q,A}}\left[F(\omega),\omega(e)\right]\\
&= \sum_{e\in E} \frac{J_e}{e^{\beta_0 J_e}-1} \Cov_{\phif_{\beta_0,q}}\left[F(\omega),\omega(e)\right],
\end{align*}
where the final equality follows by positive association.
The claim follows since $\beta_0\geq 0$ was arbitrary. 
\end{proof}

\subsection{Decision trees and the OSSS inequality}

Let $\N=\{1,2,\ldots\}$, and let $E$ be a countable set. 
  A \textbf{decision tree} is a function 
  $T:\{0,1\}^E \to E^\N$ from subsets of $E$ to infinite $E$-valued sequences with the property that
$T_1(\omega) = e_1$ for some  fixed $e_1 \in E$, and for each $n \geq 2$ there exists a function $S_n : (E\times \{0,1\})^{n-1} \to E$ such that
\[T_n(\omega) = S_n\left[\left(T_i,\omega(T_i)\right)_{i=1}^{n-1}\right].
\]
In other words, $T$ is a deterministic procedure for querying the values of $\omega \in \{0,1\}^E$, that starts by querying the value of $\omega(e_1)$ and chooses which values to query at each subsequent step as a function of the values it has already observed. 

Now let $\mu$ be a probability measure on $\{0,1\}^E$ and let $\omega$ be a random variable with law $\mu$. Given a decision tree $T$ and $n\geq 1$ we let $\cF_n(T)$ be the $\sigma$-algebra generated by the random  variables $\{T_i(\omega) : 1 \leq i \leq n\}$ and let $\cF(T) =\bigcup \cF_n(T)$. For each measurable function $f:\{0,1\}^E \to [-1,1]$, we say that $T$ \textbf{computes $f$} if $f(\omega)$ is measurable with respect to the $\mu$-completion of $\cF(T)$. By the martingale convergence theorem, if $f$  is $\mu$-integrable this is equivalent to the statement that
\[
\mu\left[f(\omega) \mid \cF_n(T)\right] \xrightarrow[n\to \infty]{} f(\omega) \qquad \mu\text{-a.s.}
\]
For each $e\in E$, we define the \textbf{revealment probability} 
\[\delta_e(T,\mu) = \mu\left(\exists n \geq 1 \text{ such that } T_{n}(\omega)=e\right).\]
Finally, following \cite{o2005every}, we define for each probability measure $\mu$ on $\{0,1\}^E$ and each pair of measurable functions $f,g:\{0,1\}^E\to \R$ the quantity
\[
\CoVr_\mu[f,g] = \mu \otimes \mu\left[|f(\omega_1)-g(\omega_2)|\right]- \mu  \left[|f(\omega_1)-g(\omega_1)|\right]
\]
where $\omega_1,\omega_2$ are drawn independently from the measure $\mu$, so that if $f$ and $g$ are $\{0,1\}$-valued then
\begin{equation}
\label{eq:CoVr_Cov}
\CoVr_\mu[f,g] = 2 \Cov_\mu[f,g] = 2\mu\bigl(f(\omega)=g(\omega)=1\bigr)-2\mu\bigl(f(\omega)=1\bigr)\mu\bigl(g(\omega)=1\bigr).
\end{equation}
We are now ready to state Duminil-Copin, Tassion, and Raoufi's generalization of the OSSS inequality to monotonic measures \cite{duminil2017sharp}.

\begin{thm}
\label{thm:OSSS}
Let $E$ be a finite or countably infinite set and let $\mu$ be a monotonic measure on $\{0,1\}^E$. Then for every pair of measurable, $\mu$-integrable functions $f,g : \{0,1\}^E \to \R$ with $f$ increasing and every decision tree  $T$ computing $g$ we have that
\[
\frac{1}{2}\left|\CoVr_\mu\left[f,g\right]\right| \leq \sum_{e\in E} \delta_e(T,\mu) \Cov_\mu\left[f,\omega(e)\right].
\]
\end{thm}


\begin{remark}
In \cite{duminil2017sharp}, only the special case of \cref{thm:OSSS} in which $E$ is finite and $f=g$ is stated. The version with $E$ finite but $f$ not necessarily equal to $g$ follows by an easy modification of their proof, identical to that carried out in \cite[Section 3.3]{o2005every} -- note in particular that when running this modified proof only $f$ is required to be increasing. The restriction that $E$ is finite can be removed via a straightforward Martingale argument \cite[Remark 2.4]{duminil2017sharp}.
\end{remark}

The statement above will be somewhat inconvenient in our analysis as the algorithm we use is naturally described as a \emph{parallel algorithm} rather than a serial algorithm. To allow for such parallelization, we define a \textbf{decision forest} to be a collection of decision trees $F=\{T^i : i \in I\}$ indexed by a countable set $I$. Given a decision forest $F=\{T^i : i \in I\}$ we let $\cF(F)$ be the smallest $\sigma$-algebra containing all of the $\sigma$-algebras $\cF(T^i)$. Given a measure $\mu$ on $\{0,1\}^E$, a function $f:\{0,1\}^E\to \R$ and a decision forest $F$, we say that $F$ \textbf{computes} $f$ if $f$ is measurable with respect to the $\mu$-completion of the $\sigma$-algebra $\cF(F)$. We also define the revealment probability $\delta_e(F,\mu)$ to be the probability under $\mu$ that there exists $i\in I$ and $n\geq 1$ such that $T^i_{n}(\omega)=e$.

\begin{corollary}
\label{cor:OSSS_forest}
Let $E$ be a finite or countably infinite set and let $\mu$ be a monotonic measure on $\{0,1\}^E$. Then for every pair of measurable, $\mu$-integrable functions $f,g : \{0,1\}^E \to \R$ with $f$ increasing and every decision forest $F$ computing $g$ we have that
\[
\frac{1}{2}\left|\CoVr_\mu\left[f,g\right]\right| \leq \sum_{e\in E} \delta_e(F,\mu) \Cov_\mu\left[f,\omega(e)\right].
\]
\end{corollary}

\begin{proof}
We may assume that $I = \{1,2,\ldots\}$. The claim may be deduced from \cref{thm:OSSS} by ``serializing'' the decision forest $F$ into a decision tree $T$. This can be done, for example, by executing the $j$th step of the decision tree $T^i$ at the time $p_i^j$ where $p_i$ is the $i$th prime, and re-querying the first input queried by $T^1$ at all times that are not prime powers. This decision tree $T$ clearly computes the same functions as the decision forest $F$ and has $\delta_e(T,\mu)=\delta_e(F,\mu)$ for every $e\in E$, so that the claim follows from \cref{thm:OSSS}.
\end{proof}

\section{Derivation of the differential inequality}

Given a graph $G=(V,E)$, a vertex $v$ and a configuration $\omega\in \{0,1\}^E$, we write $K_v=K_v(\omega)$ for the cluster of $v$ in $\omega$.

\begin{prop}
\label{eq:generalDiffIneq}
Let $G=(V,E)$ be a countable graph and let $\mu$ be a monotonic measure on $\{0,1\}^E$. Then
\begin{equation}
\sum_{e\in E} \Cov_\mu\left[\mathbbm{1}(|K_v|\geq n),\omega(e)\right]
\geq
 \left[ \myfrac[0.5em]{(1-e^{-\lambda}) - \mu\left[1-e^{-\lambda |K_v|/n}\right] }{2\sup_{u \in V}\mu\left[1-e^{-\lambda |K_u|/n}\right]}\right] \mu(|K_v|\geq n)
\end{equation}
for every $v\in V$, $n\geq 1$ and $\lambda >0 $.
\end{prop}

\begin{proof}
Let $\omega \in \{0,1\}^E$ be a random variable with law $\mu$. Independently of $\omega$, let $\eta \in \{0,1\}^V$ be a random subset of $V$ where vertices are included independently at random with inclusion probability $h=1-e^{-\lambda/n} \leq \lambda/n$. We refer to $\eta$ as the \textbf{ghost field} and call vertices with $\eta(v)=1$ \textbf{green}. Let $\P$ and $\E$ denote probabilities and expectations taken with respect to the joint law of $\omega$ and $\eta$, which is monotonic. 
Fix a vertex $v$, and let $f,g:\{0,1\}^{E \cup V} \to \{0,1\}$ be the increasing functions defined by
\[f(\omega,\eta) = \mathbbm{1}(|K_v(\omega)|\geq n) \qquad \text{ and } \qquad g(\omega,\eta)=\mathbbm{1}\left(\eta(u) =1 \text{ for some } u \in K_v(\omega)\right).\]

For each $u\in V$, we define $T^u$ to be a decision tree that first queries the status of $\eta(u)$, halts if it discovers that $\eta(u)=0$, and otherwise explores the cluster of $u$ in $\omega$. We now define this decision tree more formally. Fix an enumeration of $E$ and a vertex $u \in V$.   Set $T^u_{1}(\omega,\eta)=u$. If $\eta(u)=0$, set $T^u_{n}=u$ for every $n\geq 2$ (i.e., halt). If $\eta(u)=1$, we define $T^u_{n}(\omega,\eta)$ for $n\geq 2$ as follows. At each step of the decision tree, we will have a set of vertices $U^u_n$, a set of revealed open edges $O^u_n$, and a set of revealed closed edges $C^u_n$. We initialize by setting $U^u_1=u$ and $O^u_n=C^u_n=\emptyset$. Suppose that $n \geq 1$ and that we have computed $(U^u_k,O^u_k,C^u_k,T^u_k)$ for $k \leq n$. If every edge with at least one endpoint in $U^u_n$ is either in $O^u_n$ or $C^u_n$ then we set $(U^u_{n+1},O^u_{n+1},C^u_{n+1},T^u_{n+1})=(U^u_{n},O^u_{n},C^u_{n},T^u_{n})$ (i.e., we halt). Otherwise, we set $T^u_{n+1}$ to be the element of the set of edges that touch $U^u_n$ but are not in $O^u_n$ or $C^u_n$ that is minimal with respect to the fixed enumeration of $E$. If $\omega(T^u_{n+1})=1$ we set $U^u_{n+1}$ to be the union of $U^u_n$ with the endpoints of $T^u_{n+1}$, set $O^u_{n+1}=O^u_n \cup \{T^u_{n+1}\}$ and set $C^u_{n+1}=C^u_n$. Otherwise, $\omega(T^u_{n+1})=0$ and we set $U^u_{n+1}=U^u_n$, set $O^u_{n+1}=O^u_n$ and set $C^u_{n+1}=C^u_n \cup \{T^u_{n+1}\}$.
It is easily verified that this does indeed define a decision tree $T^u$, and that
\[
\{ x \in V \cup E : T^u_n(\omega,\eta)=x \text{ for some $n\geq 1$}\} = \begin{cases}
\{u\} & \eta(u)=0\\
\{u\} \cup E(K_u(\omega)) & \eta(u)=1,
\end{cases} 
\]
where $E(K_u(\omega))$ is the set of edges with at least one endpoint in $K_u(\omega)$.
In particular, we clearly have that the decision forest $F=\{T^u : u \in V\}$ computes $g$.

Since $f$ and $g$ are increasing and $\{0,1\}$-valued, we may apply \cref{cor:OSSS_forest} and \eqref{eq:CoVr_Cov} to deduce that
\begin{align*}
\Cov_\P[f,g] \leq \sum_{e\in E} \delta_e(F,\mu) \Cov_\P[f,\omega(e)] + \sum_{u\in V}\delta_u(F,\mu)\Cov_\P[f,\eta(v)]
= \sum_{e\in E} \delta_e(F,\mu) \Cov_\mu[f,\omega(e)],
\end{align*}
where the equality on the right follows since $f(\omega,\eta)=\mathbbm{1}(|K_v(\omega)|\geq n)$ is independent of $\eta$.
Now, an edge $e$ is revealed by $F(\omega,\eta)$ if and only if the cluster of at least one endpoint of $e$ contains a green vertex, and, writing $\eta(A)=\sum_{u \in A} \eta(u)$ for each set $A\subseteq V$, it follows that
\[\delta_e(F,\mu) \leq 2 \sup_{u\in V} \P
\left( \eta(K_u) \geq 1 \right) = 2 \sup_{u\in V} \mu\left[1-e^{-\lambda |K_u|/n}\right]\]
for every $e\in E$ and hence that
\begin{align}
\label{eq:Covproof1}
\Cov_\P[f,g] \leq 2 \sup_{u\in V} \mu\left[1-e^{-\lambda |K_u|/n}\right] \sum_{e\in E}  \Cov_\mu[f,\omega(e)].
\end{align}
To conclude, simply note that
\begin{align}
\label{eq:Covproof2}
\Cov_\P[f,g] 
 &= \P \left(|K_v| \geq n, \; \eta(K_v) \geq 1 \right)  -  \P \left(\eta(K_v) \geq 1 \right) \mu\left(|K_v| \geq n\right)\nonumber
\\
&= \mu \left[\left(1-e^{-\lambda |K_v|/n}\right)\mathbbm{1}(|K_v| \geq n)\right]  - \mu \left[1-e^{-\lambda |K_v|/n}\right] \mu\left(|K_v| \geq n\right)\nonumber\\
 &\geq (1-e^{-\lambda}) \mu\left(|K_v| \geq n\right) - \mu \left[1-e^{-\lambda |K_v|/n}\right] \mu\left(|K_v| \geq n\right).
\end{align}
Combining \eqref{eq:Covproof1} and \eqref{eq:Covproof2} and rearranging yields the desired inequality.
\end{proof}

\begin{proof}[Proof of \cref{cor:FKdiffineq}]
This is immediate from \cref{prop:DiniDerivativeFormula,eq:generalDiffIneq} together with the 
 inequality $1-e^{-\lambda |K_v|/n} \leq 1 \wedge \frac{\lambda|K_v|}{n}$, which yields the bound
\[\phih_{\beta,q}\left[1-e^{-\lambda |K_v|/n}\right] \leq \frac{\lambda}{n}\,\phih_{\beta,q}\!\left[\frac{n}{\lambda} \wedge |K_v| \right] \leq \frac{\lambda}{n}\sum_{m=1}^{\lceil n/\lambda \rceil} \phih_{\beta,q}(|K_v|\geq m).
\qedhere\]
\end{proof}

Taking the limit as $\lambda \downarrow 0$ in \cref{cor:FKdiffineq} yields the  following corollary.

\begin{corollary}
\label{cor:FKdiffineq2}
Let $(G,J)$ be an infinite transitive weighted graph, and let $q\geq 1$ and $\# \in \{\mathrm{f},\mathrm{w}\}$. Then
\begin{equation}
\label{eq:FKdiffineq}
\max_{e\in E} \left[\frac{e^{\beta J_e}-1}{J_e}\right]
   \lrDini{\beta} \log \phih_{\beta,q}\left(|K|\geq n\right) 
\geq \frac{1}{2} \left[ \frac{n}{\phih_{\beta,q}[|K|]}-1\right]
\end{equation}
for every $n\geq 1$ and $\beta \geq 0$.
\end{corollary}


Since $\phih_{\beta,q}(|K|\geq n)$ is increasing in $\beta$, the following inequalities may be obtained by integrating the differential inequalities of \cref{cor:FKdiffineq,cor:FKdiffineq2}:   
Letting $C(\beta)=\max_{e\in E} \left[\frac{e^{\beta J_e}-1}{J_e}\right]$ for each $\beta \geq 0$, we have that
\begin{equation}
\label{eq:integrated1}
\phih_{\beta,q}(|K|\geq n) \leq \phih_{\beta_0,q}(|K|\geq n) \exp\left[ - \frac{(1-e^{-1})(\beta_0-\beta)n}{2 C(\beta_0) \sum_{m=1}^{n} \phih_{\beta_0,q}(|K|\geq m)} + \frac{\beta_0-\beta}{2C(\beta_0)}\right]
\end{equation}
and
\begin{equation}
\label{eq:integrated2}
\phih_{\beta,q}(|K|\geq n) \leq \phih_{\beta_0,q}(|K|\geq n) \exp\left[ - \frac{(\beta_0-\beta)n}{2 C(\beta_0) \phih_{\beta_0,q}\bigl[|K|\bigr]} + \frac{\beta_0-\beta}{2C(\beta_0)}\right]
\end{equation}
for every $n\geq 1$, $0 \leq \beta \leq \beta_0$, and $q\geq 1$.
\section{Analysis of the differential inequality}

\subsection{Critical exponent inequalities}

In this section we apply \cref{cor:FKdiffineq} to prove \cref{thm:gammadelta_percolation,thm:gammagap_percolation,thm:gammadelta_FK,thm:gammagap_FK}.

\begin{proof}[Proof of \cref{thm:gammadelta_percolation,thm:gammadelta_FK}]
Fix $\beta_0 > 0$, and suppose that there exist constants $C>0$ and $\delta>1$ such that $\phih_{\beta_0,q}(|K|\geq n) \leq Cn^{-1/\delta}$ for every $n\geq 1$.  In this proof, we use $\preceq$ and $\succeq$ to denote inequalities that hold up to posiitve multiplicative constants depending only on $(G,J)$, $\delta$, $C$, and $\beta_0$.  We have that
\[
\sum_{m=1}^{n}\phih_{\beta_0,q}(|K|\geq m) \preceq n^{1-1/\delta}
\]
for every $n\geq 1$ and hence by \eqref{eq:integrated1} that
\begin{equation*} \phih_{\beta_1,q}\left(|K|\geq n\right)  \preceq n^{-1/\delta}  \exp\left[ -c_1 (\beta_0-\beta_1) n^{1/\delta}   \right]
 \end{equation*}
for every $0\leq \beta_1 < \beta_0$ and $n\geq 1$. This inequality may be summed over $n$ to obtain that
\begin{equation}
\label{eq:momentboundproof}
\phih_{\beta_1,q}\left[|K|\right] \preceq \sum_{n \geq 1} n^{-1/\delta} \exp\left[ -c_1 (\beta_0-\beta_1) n^{1/\delta}   \right] \preceq (\beta_0-\beta_1)^{-\delta+1}
 \end{equation}
 for every $0 \leq \beta_1 < \beta_0$ and $n\geq 1$. Thus, we have that
 \begin{equation*}
\phih_{\beta,q}(|K|\geq n)  
\preceq n^{-1/\delta}\exp\left[-c_2 \frac{(\beta_1-\beta)n}{(\beta_0-\beta_1)^{-\delta+1}}\right]
 \end{equation*}
for every $0 \leq \beta < \beta_1 < \beta_0$ and $n\geq 1$. Item 1 of the theorem follows by taking $\beta_1 = (\beta_0+\beta)/2$. 


Item 2 of the theorem is a simple analytic consequence of item 1 since we have that 
\[\phih_{\beta,q}(|K|\geq x) = \phih_{\beta,q}(|K|\geq \lceil x \rceil) \preceq x^{-1/\delta} \exp\left[-c_2 (\beta_0-\beta)^\delta n\right]\] for every $x>0$, and consequently that, letting $\eps=c_2(\beta_0-\beta)^\delta$ and $\alpha=k-1-1/\delta$, we have that
\begin{multline}
\phih_{\beta,q}\left[|K|^k\right] = k\int_0^\infty x^{k-1} \phih_{\beta,q}(|K|\geq x) \dif x  \preceq k \int_{x= 0}^\infty x^{\alpha} e^{-\eps x} \dif x = k \eps^{-{\alpha}-1}\int_{y= 0}^\infty y^{\alpha} e^{ -y} \dif y \\=  k \Gamma({\alpha}+1) \eps^{-{\alpha}-1} \leq k! \eps^{-{\alpha}-1}  
= k! \left[c_2(\beta_0-\beta)\right]^{-\delta(k-1)-(\delta-1)}
\label{eq:moment_calculation}
\end{multline}
for every $k\geq 1$ and $0\leq \beta < \beta_0$, where we used the change of variables $y=\eps x$ in the final equality on the first line.
\end{proof}

\begin{proof}[Proof of \cref{thm:gammagap_percolation,thm:gammagap_FK}]
This proof is similar to that of \cref{thm:gammadelta_FK}. Fix $\beta_0 > 0$, and suppose that there exist constants $C>0$ and $\delta>1$ such that $\phih_{\beta,q}(|K|\geq n) \leq C(\beta_0-\beta)^{-\gamma}$ for every $0\leq \beta < \beta_0$.  In this proof, we use $\preceq$ and $\succeq$ to denote inequalities that hold up to posiitve multiplicative constants depending only on $(G,J)$, $\gamma$, $C$, and $\beta_0$. By \eqref{eq:integrated2} and Markov's inequality there exist positive constants $c_1$ and $c_2$ such that 
\[
\phih_{\beta,q}(|K|\geq n) \preceq \frac{\phih_{\beta_1,q}[|K|]}{n}\exp\left[-c_1 \frac{n(\beta_1-\beta_0)}{\phih_{\beta_1,q}[|K|]}n\right] \preceq \frac{1}{(\beta_0-\beta)^{\gamma}n}\exp\left[-c_2(\beta_1-\beta_0)^{\gamma+1}n\right]
\]
for every $0 \leq \beta < \beta_1 < \beta_0$ and $n\geq 1$. The proof may be concluded by  taking $\beta_1=(\beta_0+\beta)/2$ and performing essentially the same calculation as in \eqref{eq:moment_calculation}.
\end{proof}

\subsection{Sharpness of the phase transition}

We next apply \cref{cor:FKdiffineq} to prove \cref{thm:sharpness}. The proof is very similar to that given in \cite{duminil2017sharp}. (Note that the analysis presented there substantially simplified the original analysis of Menshikov \cite{MR852458}.) We include it for completeness since our differential inequality is slightly different, and so that we can optimize the constants appearing in item 2 of \cref{thm:sharpness}.

\begin{proof}[Proof of \cref{thm:sharpness}]
 We  begin by defining 
\begin{align*}
\tilde \beta_c^\hash &= \sup\Bigl\{ \beta \geq 0 : \text{ there exists $c,C>0$ such that } \phih_{\beta,q}(|K| \geq n) \leq C n^{-c} \text{ for every $n\geq 1$}\Bigr\}\\
&= \inf\biggl\{\beta \geq 0 : \limsup_{n\to\infty} \frac{\log \phih_{\beta,q}(|K|\geq n)}{\log  n} \geq 0 \biggr\}.
\end{align*}
We trivially have that $\tilde \beta_c^\hash \leq \beta_c^\hash$. Moreover, it is an immediate consequence of \cref{thm:gammadelta_FK} that for every $0\leq \beta< \tilde\beta_c^\hash$ there exists $C_\beta,c_\beta>0$ such that
\begin{equation}
\label{eq:sharp1}
\phih_{\beta,q}(|K| \geq n) \leq C_\beta e^{-c_\beta n} \qquad\text{for every $n\geq 1$}.
\end{equation}

We next claim that $\phih_{\beta,q}(|K|=\infty)>0$ for every $\beta>\tilde \beta_c^\hash$.
To this end, write $P_n(\beta)=\phih_{\beta,q}(|K| \geq n)$ and $\Sigma_n(\beta)=\sum_{m=0}^{n-1}P_m(\beta)$ and let 
\[T_k(\beta)=\frac{1}{\log k}\sum_{n=1}^k \frac{1}{n}\phih_{\beta,q}(|K| \geq n)\]
for each $\beta\geq 0$ and $k\geq 2$, so that $\lim_{k\to\infty} T_k(\beta) = \phih_{\beta,q}(|K| = \infty)$ for each $\beta \geq 0$. Applying \cref{cor:FKdiffineq} with $\lambda=1$, we obtain that
\[
\lrDini{\beta} T_k(\beta) \geq \frac{1}{2C(\beta) \log k} \sum_{n=1}^k \left[\frac{(1-e^{-1})P_{n}(\beta)}{\Sigma_n(\beta)} - \frac{P_n(\beta)}{n} \right]
\] 
for every $\beta \geq 0$ and $k\geq 2$.
Using the inequality
$\frac{P_n}{\Sigma_n} \geq \int_{\Sigma_n}^{\Sigma_{n+1}}\frac{1}{x}\dif x = \log \Sigma_{n+1} - \log \Sigma_n
$, 
we deduce that
\[
\lrDini{\beta} T_k(\beta) \geq \frac{(1-e^{-1})\log \Sigma_{k+1}(\beta)}{2C(\beta) \log k}  - \frac{T_k(\beta)}{2C(\beta)} 
\] 
for every $\beta \geq 0$ and $k\geq 1$. 
 Fixing $\tilde \beta_c^\hash < \beta_1 < \beta_2$, we deduce that
\[
\lrDini{\beta} T_k(\beta) \geq \frac{(1-e^{-1})\log \Sigma_{k+1}(\beta_1)}{2C(\beta_2) \log k}  - \frac{T_k(\beta_2)}{2C(\beta_2)} 
\] 
for every $\beta_1 \leq \beta \leq \beta_2$ and hence by definition of $\tilde \beta_c^\hash$ that
\[
\limsup_{k\to\infty} \inf_{\beta_1 \leq \beta \leq \beta_2} \lrDini{\beta} T_k(\beta) \geq \frac{(1-e^{-1})}{2C(\beta_2)}  - \frac{\phih_{\beta_2,q}(|K|=\infty)}{2C(\beta_2)} .
\] 
Integrating this inequality yields that
\[
\phih_{\beta_2,q}(|K|=\infty) \geq \limsup_{k\to\infty} \int_{\beta_1}^{\beta_2} \lrDini{\beta} T_k(\beta) \dif \beta \geq \frac{\beta_2-\beta_1}{2C(\beta_2)}\left[1-e^{-1}-\phih_{\beta_2,q}(|K|=\infty)\right],
\]
which rearranges to give that 
\[
\phih_{\beta_2,q}(|K|=\infty) \geq  \frac{(1-e^{-1})(\beta_2-\beta_1)}{2C(\beta_2)+\beta_2-\beta_1}>0.
\]
The claim now follows since $\tilde  \beta_c^\hash<\beta_1 <\beta_2$ were arbitrary. We deduce that $\tilde \beta_c^\hash \geq \beta_c^\hash$ and hence that  $\tilde \beta_c^\hash = \beta_c^\hash$, so that in particular item 1 of the theorem follows from \eqref{eq:sharp1}.

Finally, we apply \cref{cor:FKdiffineq} again to obtain that
\[
\lrDini{\beta} T_k(\beta) \geq \frac{1}{2C(\beta) \log k} \sum_{n=1}^k \left[\frac{(1-e^{-\lambda})P_{n}(\beta)}{\lambda \Sigma_{\lfloor n/\lambda \rfloor}(\beta)} - \frac{P_n(\beta)}{n} \right]
\]
for  each $k\geq 2$  and  $\beta \geq 0$. Arguing similarly before, we obtain that
\[
\limsup_{k\to\infty} \inf_{\beta_1 \leq \beta \leq \beta_2}  \lrDini{\beta} T_k(\beta) \geq \frac{1-e^{-\lambda }}{2C(\beta_2)} - \frac{\phih_{\beta,q}(|K|=\infty)}{2C(\beta_2)}
\]
for every $\beta_c^\hash < \beta_1 \leq \beta_2$,  and item 2 of the theorem follows by sending $\lambda \to\infty$ and then integrating the resulting inequality.  
\end{proof}



\subsection*{Acknowledgments} 
We thank Hugo Duminil-Copin and Geoffrey Grimmett for helpful discussions.

  \setstretch{1}
  \bibliographystyle{abbrv}
{\small
  \bibliography{unimodularthesis.bib}
}
\end{document}